\documentclass[12pt]{amsart}
\usepackage{latexsym, amsmath, amssymb, amsthm}
\usepackage{epsfig}
\usepackage{amscd}
\usepackage{latexsym}
\usepackage{pstricks,pst-node,cite}

\makeatletter
\newtheorem*{rep@theorem}{\rep@title}
\newcommand{\newreptheorem}[2]{%
\newenvironment{rep#1}[1]{%
 \def\rep@title{#2 \ref{##1}}%
 \begin{rep@theorem}}%
 {\end{rep@theorem}}}
\makeatother

\newtheorem{theorem}{Theorem}[section]
\newreptheorem{theorem}{Theorem}
\newtheorem{lemma}[theorem]{Lemma}

\newtheorem{corollary}[theorem]{Corollary}
\newreptheorem{cor}{Corollary}
\newtheorem{example}[theorem]{Example}

\newtheorem{definition}[theorem]{Definition}

\makeatletter
\providecommand*{\cupdot}{%
  \mathbin{%
    \mathpalette\@cupdot{}%
  }%
}
\newcommand*{\@cupdot}[2]{%
  \ooalign{%
    $\m@th#1\cup$\cr
    \sbox0{$#1\cup$}%
    \dimen@=\ht0 %
    \sbox0{$\m@th#1\cdot$}%
    \advance\dimen@ by -\ht0 %
    \dimen@=.5\dimen@
    \hidewidth\raise\dimen@\box0\hidewidth
  }%
}

\providecommand*{\bigcupdot}{%
  \mathop{%
    \vphantom{\bigcup}%
    \mathpalette\@bigcupdot{}%
  }%
}
\newcommand*{\@bigcupdot}[2]{%
  \ooalign{%
    $\m@th#1\bigcup$\cr
    \sbox0{$#1\bigcup$}%
    \dimen@=\ht0 %
    \advance\dimen@ by -\dp0 %
    \sbox0{\scalebox{2}{$\m@th#1\cdot$}}%
    \advance\dimen@ by -\ht0 %
    \dimen@=.5\dimen@
    \hidewidth\raise\dimen@\box0\hidewidth
  }%
}
\makeatother

\newcommand{\F}{\mathcal{F}}

\newcommand{\BS}{\mathbb{S}}

\newcounter{fignum}
\ifx\blackandwhite\undefined
  \newrgbcolor{darkred}{.9 0 0}\newrgbcolor{emgreen}{0 .9 0}
  \newrgbcolor{pup}{.7  0 .9}\newrgbcolor{xxx}{.8 .8 .8}
  \newrgbcolor{lightgray}{.99 .99 0}\newrgbcolor{darkgray}{.8 .8 .3} \else
\fi
\psset{linecolor=blue}

\psset{linewidth=.7pt,dash=3pt 3pt,doublesep=.05,dotsize=1pt 5}
\SpecialCoor

\begin{document}

\title[Knots of the flat plumbing basket number $6$]{The complete list of prime knots whose flat plumbing basket numbers are $6$ or less}

\author{Yoon-Ho Choi}
\address{School of Computer Science \& Engineering \\ Pusan National University
\\ Pusan, 609-735 Korea}
\email{choi.yuno@gmail.com}

\author{Yun Ki Chung}
\address{Gyeonggi Science High School \\Suwon, 440-800, Korea}
\email{jyg9628@nate.com}

\author{Dongseok Kim}
\address{Department of Mathematics \\Kyonggi University
\\ Suwon, 443-760 Korea}
\email{dongseok@kgu.ac.kr}



\maketitle

\begin{abstract}
Flat plumbing basket surfaces of links were introduced
to study the geometry of the complement of the links.
These flat plumbing basket surface can be presented by a sequential presentation
known as flat plumbing basket code first found by Furihata, Hirasawa and Kobayashi.
The minimum number of flat plumbings to obtain a
flat plumbing basket surfaces of a link
is defined to be \emph{the flat plumbing basket number}
of the given link.
In present article, we use these sequential presentations to find the
complete classification theorem of prime knots whose flat plumbing basket number $6$ or less.
As applications, this result improves the work of Hirose and Nakashima
which finds the flat plumbing basket number of prime knots up to $9$ crossings.
\end{abstract}



\section{Introduction}

Orientable surfaces whose boundary is the given link,
known as \emph{Seifert surfaces} have been studied
for many interesting invariants of links such as Seifert pairings, Alexander polynomials, signatures and etc.
A \emph{plumbing surface} obtained from a $2$-dimensional disc by plumbings annuli found by
Rudolph~\cite{Rudolph:plumbing} used to
study extensively for the fibreness of links and surfaces
~\cite{FHK:openbook, Gabai:murasugi1, Gabai:murasugi2, harer:const, Nakamura,
Rudolph:quasipositive2, Stallings:const}.
In particular, if we only use flat annuli plumbings, the resulting surface is
called a \emph{flat plumbing surface}. The main focus of the present article is
flat plumbing basket surfaces, a precise definition can be found in Definition~\ref{defi1}.
A flat plumbing basket surface can be regarded as a flat plumbing surface, but not
vice versa. There exists a Seifert surface which is obtained from a disk by successively
plumbing flat annuli, but which is
not isotopic to any flat plumbing basket surface~\cite{FHK:openbook}.

The third author's first preprint about these plumbing surfaces from a canonical
Seifert surface had a critical mistake. In the process of resolving this mistake,
the third author, Kwon and Lee proved the existence of banded surfaces and flat banded surfaces~\cite{KKL:string}
by weakening some conditions of plumbings.
The third author also proved that every link $L$ is the boundary of an oriented surface
which is obtained from a graph embedding of a dipole graph,
this surface is also known as a braidzel surface~\cite{Nakamura:braidzel},
and a complete bipartite graph $K_{2,n}$,
where all voltage assignments on the edges of dipole graph and $K_{2,n}$ are $0$~\cite{Kim:dipole}.
The mistake was finally fixed in~\cite{Kim:flat}.

The present work is one of articles in this series of results presenting
links as a boundary of the surface obtained in a
embedding of certain graphs as described in~\cite{GT1}. One might consider
these plumbing surfaces as special embeddings of the bouquets of circles~\cite{GRT}.

A sequence of articles by Furihata, Hirasawa and Kobayashi \cite{FHK:openbook} and the third author~\cite{Kim:flat} proved
the existence of a flat plumbing basket surface of a given link $L$.
We can define the \emph{flat plumbing basket number} of $L$, denoted
by $fpbk(L)$, to be the minimal number of flat annuli to obtain a
flat plumbing basket surface of the link $L$.
However, finding the flat plumbing basket number of a link $L$ is very difficult and has been
beyond the reach for more than $15$ years from the first invention by
Rudolph~\cite{Rudolph:plumbing}.
because it is defined to be the minimum over all possible
flat plumbing basket surfaces whose boundaries is the given link $L$.

The work of Furihata \emph{et al.}~\cite{FHK:openbook} provided not only the existence theorem
using a very tangible alternating definition
of the flat plumbing basket surface but also
a coding algorithm, the resulting code is called
\emph{flat plumbing basket code}, to present links
as the boundaries of flat plumbing basket surfaces from a special
closed braid presentation of the link.

In present article, we use these sequential codes to find
all prime knots of the flat plumbing basket number $6$ in Theorem~\ref{6classification}
by applying DT-code and a computer
program ``knotfinder" and ``knotscape" of {\tt{Knotscape}}~\cite{Thistlethwaite:knotscape}.

When the third author first presented this work at the TAPU conference in 2013 summer,
Carter pointed out that this flat plumbing basket code
of a link can be very useful to calculate Alexander polynomials
because all components in Seifert matrix can be found directly from the
presentation and they are either $0$ or $\pm1$. A very recent work by
Hirose and Nakashima~\cite{HN} found a theorem which provide two lower bounds
of the flat plumbing basket number using Alexander polynomials and genera of links.
Using these lower bounds, they succeed to find the flat plumbing basket
number of all prime knots up to $9$ crossings except $24$ knots.

As an application of our classification theorem,
we find the flat plumbing basket number of
five knots out of $24$ knots and sharpens the range of
the flat plumbing basket number of three knots.

The outline of this paper is as follows. We first provide
some preliminary definitions and results in Section~\ref{prelim}.
We provided an explicit coding algorithm to find the flat plumbing basket
presentation of a link from its braid presentation and
canonical Seifert surface.
Also we provide two classification theorems of the flat plumbing basket number of $4$
and $6$ with a explanation how we find DT-code and use the computer
program ``knotfinder" of {\tt{Knotscape}} in Section~\ref{result}.
We conclude with a remark on further research in Section~\ref{conclusion}.

\section{Preliminaries} \label{prelim}

A compact orientable surface $\F$ is called a \emph{Seifert surface}
of a link $L$ if the boundary of $\F$ is isotopic to the given link
$L$. The existence of such a surface was first proven by Seifert
using an algorithm on a diagram of $L$, this algorithm was named after him as
\emph{Seifert's algorithm}~\cite{Seifert:def}.
A Seifert surface $\F_L$ of an oriented link $L$ produced by applying Seifert's
algorithm to a link diagram is called a \emph{canonical Seifert surface}.

\begin{figure}
$$
\begin{pspicture}[shift=-1.6](-.2,-2.2)(5,1.7)
\psline[linewidth=2pt](0,-1.5)(.8,-.5)
\psline[linecolor=darkgray, linewidth=1.5pt](1.2,0)(4.5,0)
\psline[linestyle=dashed, linewidth=2pt](.8,-.5)(2,1)
\psline[linewidth=2pt](2.4,1.5)(2,1)
\pccurve[angleA=110,angleB=180](.8,-.5)(1.4,1)
\pccurve[linestyle=dashed, angleA=0,angleB=90](1.4,1)(1.6,.5)
\psline(1.4,1)(4.4,1)
\pccurve[angleA=110,angleB=180](4.1,-.5)(4.4,1)
\pccurve[angleA=0,angleB=90](4.4,1)(4.9,.5)
\psline[linewidth=2pt](3.3,-1.5)(5.7,1.5)
\psline(.8,-.5)(4.1,-.5) \psline[linestyle=dashed](1.6,.5)(4.9,.5)
\pccurve[angleA=200,angleB=0](5.7,1.5)(4.6,1.7)
\pccurve[angleA=180,angleB=0](4.6,1.7)(3.5,1.3)
\pccurve[angleA=180,angleB=-20](3.5,1.3)(2.4,1.5)
\pccurve[angleA=200,angleB=0](3.3,-1.5)(2.2,-1.7)
\pccurve[angleA=180,angleB=0](2.2,-1.7)(1.1,-1.3)
\pccurve[angleA=180,angleB=-20](1.1,-1.3)(0,-1.5)
\rput(3,.2){$\alpha$} \rput(2.5,-.28){$C_{\alpha}$}
\rput(2,-1){$S$} \rput(1.25,.6){$B_{\alpha}$}
\rput[t](2.7,-2){$(a)$}
\end{pspicture}\quad
\begin{pspicture}[shift=-1.6](.8,-2.2)(6,1.7)
\psline[linewidth=2pt](0,-1.5)(.8,-.5)
\psline(.8,-.5)(1.6,.5) \psline[linewidth=2pt](1.6,.5)(2.4,1.5)
\pccurve[linewidth=2pt, angleA=45,angleB=135](1.6,.5)(4.9,.5)
\pccurve[linewidth=2pt, angleA=45,angleB=135](.8,-.5)(4.1,-.5)
\psline[linewidth=2pt](3.3,-1.5)(4.1,-.5)
\psline(4.1,-.5)(4.9,.5) \psline[linewidth=2pt](4.9,.5)(5.7,1.5)
\psline[linestyle=dotted](.8,-.5)(4.1,-.5)
\psline[linestyle=dotted](1.6,.5)(4.9,.5)
\pccurve[angleA=200,angleB=0](5.7,1.5)(4.6,1.7)
\pccurve[angleA=180,angleB=0](4.6,1.7)(3.5,1.3)
\pccurve[angleA=180,angleB=-20](3.5,1.3)(2.4,1.5)
\pccurve[angleA=200,angleB=0](3.3,-1.5)(2.2,-1.7)
\pccurve[angleA=180,angleB=0](2.2,-1.7)(1.1,-1.3)
\pccurve[angleA=180,angleB=-20](1.1,-1.3)(0,-1.5)
\psline(4.2,.2)(5.2,.2) \psline[arrowscale=1.5]{->}(5.18,.2)(5.2,.2)
\rput(2.5,-.28){$C_{\alpha}$} \rput(5.5,.2){$A_{0}$}
\rput(2,-1){$\overline{S}$}
\rput[t](2.7,-2){$(b)$}
\end{pspicture}
$$
\caption{$(a)$ A geometric shape of $\alpha, B_{\alpha}$
and $C_{\alpha}$ on a Seifert surface $S$
and $(b)$ a new Seifert surface $\overline{S}$ obtained
from $S$ by a top $A_0$ plumbing along the path
$\alpha$.}
\label{topfig}
\end{figure}
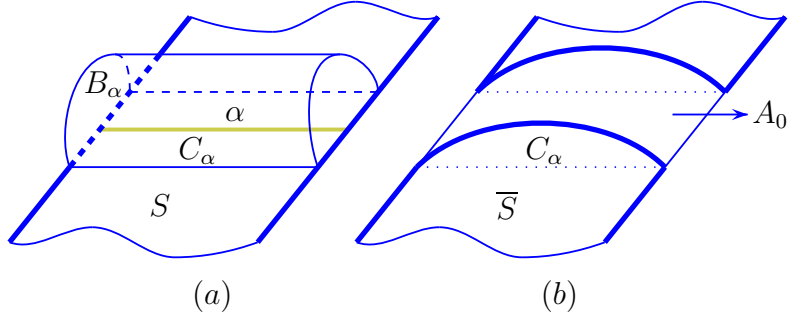

The main topic of the article is the flat plumbing basket surfaces.
Rudolph~\cite{Rudolph:plumbing} first defined the top plumbing as follows.
Let $\alpha$ be a proper arc on a Seifert surface $S$. Let $B_{\alpha}$ be a $3$-cell on
the positive side of the oriented surface $S$ along a tubular neighborhood $C_{\alpha}$ of $\alpha$ on $S$.
Let $A_n\subset B_{\alpha}$ be an $n$ times full twisted annulus such that
$A_n \cap \partial B_{\alpha}= C_{\alpha}$. The \emph{top plumbing} on $S$ along a path $\alpha$
is the new surface $S'= S \cup C_{\alpha}$ where $A_n, B_{\alpha}, C_{\alpha}$ satisfy
the previous conditions as depicted in Fig.~\ref{topfig}.
Thus, two consecutive plumbings are non-commutative in general.
Rudolph found a few interesting results with regards to the top and bottom
plumbings in~\cite{Rudolph:plumbing}. For the rest of article, all plumbings are top plumbing unless state differently.

\begin{definition} \label{defi1}
A Seifert surface $\F$ is a \emph{flat plumbing basket surface} if
$\F = D_2$ or if $\F = \F_0
*_{\alpha} A_0$ which can be constructed by plumbing $A_0$ to a flat plumbing
basket surface $\F_0$ along a proper arc $\alpha \subset D_2\subset \F_0$.
We say that a link $L$ admits a \emph{flat plumbing basket
representation} if there exists a flat plumbing basket surface $\F$ such that
$\partial \F$ is equivalent to $L$.
\end{definition}

An alternative definition of the flat plumbing basket
surfaces is given in~\cite{FHK:openbook}
and it is very easy to follow. The \emph{trivial
open book decomposition} of $\mathbb{R}^3$
is a decomposition of $\mathbb{R}^3$ into the half planes
in the following form. In a cylindrical coordinate, it can be presented
$$ \mathbb{R}^3 = \bigcup_{\theta \in [0, 2\pi)}
\{(r, \theta, z) | r \ge 0, z \in \mathbb{R} \}$$
where $\{(r, \theta, z) | r \ge 0, z \in \mathbb{R} \}$
is called a \emph{page} for $\theta \in [0, 2\pi)$.
Let $\mathcal{O}$ be the \emph{trivial open book decomposition} of the
$3$-sphere $\BS^3$ which is obtained from the trivial
open book decomposition of $\mathbb{R}^3$
by the one point compactification. A Seifert surface
is said to be a flat plumbing basket surface
if it consists of a single page of $\mathcal{O}$ as a $2$-disc $D^2$ and
finitely many bands which are
embedded in distinct pages~\cite{FHK:openbook}.
Flat plumbing basket surfaces
of $(i)$ the trefoil knot and $(ii)$ the figure
eight knot in the trivial open book decomposition are depicted
in Fig.~\ref{figure84band} where $D^2$ is presented as a shaded
rectangular region and the top horizontal line of
the rectangle is in the $z$-axis and the top
hemi-spherical annuli are contained in different pages.

\begin{figure}
$$
\begin{pspicture}[shift=-1.2](-.7,-1.8)(4.2,1.2)
\psarc[doubleline=true](2.5,0){1}{-5}{185}
\psarc[doubleline=true](2,0){1}{-5}{185}
\psarc[doubleline=true](1.5,0){1}{-5}{185}
\psarc[doubleline=true](1,0){1}{-5}{185}
\psframe[linecolor=lightgray,fillstyle=solid,fillcolor=lightgray](-.5,-1)(4,0)
\psline(-.03,0)(-.5,0)(-.5,-1)(4,-1)(4,0)(3.53,0)
\psline(.03,0)(.47,0) \psline(.53,0)(.97,0) \psline(1.03,0)(1.47,0)
\psline(1.53,0)(1.97,0) \psline(2.03,0)(2.47,0) \psline(2.53,0)(2.97,0)
\psline(3.03,0)(3.47,0)
\rput(1.75,-.5){{$\mathcal{D}$}}
\rput(1.75,-1.5){{$(a)$}}
\end{pspicture}
\hskip 1cm
\begin{pspicture}[shift=-1.2](-.7,-1.8)(4.2,1.2)
\psarc[doubleline=true](2,0){1}{-5}{185}
\psarc[doubleline=true](2.5,0){1}{-5}{185}
\psarc[doubleline=true](1.5,0){1}{-5}{185}
\psarc[doubleline=true](1,0){1}{-5}{185}
\psframe[linecolor=lightgray,fillstyle=solid,fillcolor=lightgray](-.5,-1)(4,0)
\psline(-.03,0)(-.5,0)(-.5,-1)(4,-1)(4,0)(3.53,0)
\psline(.03,0)(.47,0) \psline(.53,0)(.97,0) \psline(1.03,0)(1.47,0)
\psline(1.53,0)(1.97,0) \psline(2.03,0)(2.47,0) \psline(2.53,0)(2.97,0)
\psline(3.03,0)(3.47,0)
\rput(1.75,-.5){{$\mathcal{D}$}}
\rput(1.75,-1.5){{$(b)$}}
\end{pspicture}
$$
\caption{$(a)$ A flat $4$-banded surface of the trefoil knot and $(b)$ a flat $4$-banded surface of the figure eight knot.} \label{figure84band}
\end{figure}
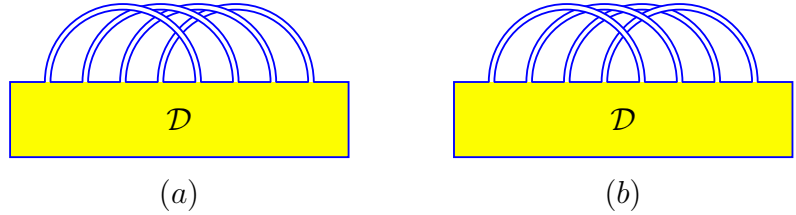

Using this definition, for a given link $L$, Furihata \emph{et al.} \cite{FHK:openbook} found an algorithm to find
a flat plumbing basket surface from a closed braid $\overline{\beta}=L$. So
we can define the \emph{flat plumbing basket number} of $L$, denoted
by $fpbk(L)$, to be the minimal number of flat annuli to obtain a
flat plumbing basket surface of $L$.

\begin{theorem} (\cite{FHK:openbook})
Let $L$ be an oriented link which is a closed $n$-braid with a braid
word $\sigma_{n-1}\sigma_{n-2}\ldots\sigma_1 W$ where the length of
$W$ is $m$ and $W$ has $p$ positive letters, then there exists a
flat plumbing basket surface $S$ with $m + 2p$ bands such that
$\partial S$ is isotopic to $L$, $i.e.,$ $fpbk(L)\le m+2p$.
\label{falbktheorem1}
\end{theorem}

This upper bound has been improved by the third author~\cite{Kim:flat}
where the link is prime but not splittable.

\begin{theorem} (\cite{Kim:flat})
Let $L$ be an oriented link which is a closed $n$-braid with a braid
word $\beta$ whose length is $m$ and
let $ps(\sigma_i^{\pm 1})$ be the power sum
of $\sigma_i^{\pm 1}$ in $\beta$ for all $i=1, 2,
\ldots, n-1$.
Let $\gamma$ be the cardinality of the set
$$\Omega= \{ i | 1
\le i \le n-1, \sigma_i ~{\rm{and}}~ \sigma_i^{-1}~ {\rm{both}~\rm{appear}~ \rm{in}}~ \beta\}.$$
Let
$$\epsilon_i = \begin{cases} 1 ~~~& {\rm{if}}~ 1 \le
ps(\sigma_i^{1}) \le ps(\sigma_i^{- 1}) ~{\rm{or}}~ps(\sigma_i^{- 1})=0, \\
-1 ~~~&{\rm{if}}~ 1 \le ps(\sigma_i^{-1}) \le ps(\sigma_i^{1})
~{\rm{or}}~ps(\sigma_i^{1})=0. \end{cases}
$$
Then the flat plumbing basket number of $L$ is bounded by
$m + n-1 -4\gamma +2\sum_{i=1}^{n-1} ps(\sigma_i^{\epsilon_i})$, $i. e.,$
$$fpbk(L)\le m + n-1 -4\gamma +2\sum_{i=1}^{n-1} ps(\sigma_i^{\epsilon_i}).$$
\label{fpbktheoremext}
\end{theorem}

The third author proved that every link $L$ admits a flat plumbing basket representation
from a canonical Seifert surface $\F_L$ of $L$ to have a property that the Seifert graph
$\Gamma(D_L)$ has a co-tree edge alternating spanning tree $T$~\cite{Kim:flat}.

\begin{theorem} (\cite{Kim:flat})
Let $\Gamma$ be an Seifert graph of canonical Seifert surface $S$ of a link $L$ with
$|V(\Gamma)|=n$, $|E(\Gamma)|=m$ and the sign labeling $\phi$.
Let $G(\Gamma)$ be the Seifert graph of $\Gamma$.
Let $T$ be a co-tree  edge alternating spanning tree of $\Gamma$ and $\mu$ a
labeling on $T$ chosen in~\cite[Theorem 3.3]{Kim:flat}.
Let $\delta(T)$ be the cardinality of the set
$$\Psi(T) =\{e\in E(T)|~\mu(e)\neq\phi(\overline{e}) ~{\rm{for~all}}~ \overline{e} \in \Gamma(e) \},$$
and let $\zeta(T)$ be the cardinality of the set
$$\Upsilon(T) =\{\overline{e}\in E(\Gamma(T))~|~\mu(e)=\phi(\overline{e}),~\overline{e}\in \Gamma(e),~e \in E(T)-\Psi(T) \}.$$
and let $\eta(T)$ be the cardinality of the set
$$\Phi(T) = \{ \overline{e}\in E(\Gamma)-E(\Gamma(T)) ~|~\mu(\overline{e}) = \nu(e) \}$$
where $\nu(e) =+$($-$, resp.) if there is one extra positive(negative, respectively) sign
in the path $P_e$ joining end vertices of the edge $e$ in $T$.
Then the
flat plumbing basket number of $L$ is bounded by $m - 3(n-1) + 2 ( 2\delta(T)+\zeta(T) + \eta(T))$, $i. e.$,
$$ fpbk(L)\le m - 3(n-1) + 2 ( 2\delta(T)+\zeta(T) + \eta(T)).$$ \label{flatbktheorem3}
\end{theorem}

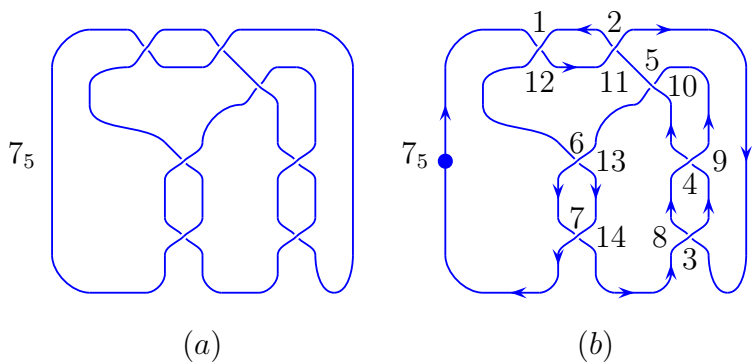
\begin{figure}
$$
\begin{pspicture}[shift=-2](-.2,-.7)(4.2,3.7)
\psline(0,.5)(0,3)
\pccurve[angleA=90,angleB=180](0,3)(.5,3.5)
\psline(.5,3.5)(1,3.5)
\pccurve[angleA=0,angleB=135](1,3.5)(1.22,3.28)
\pccurve[angleA=-45,angleB=180](1.28,3.22)(1.5,3)
\psline(1.5,3)(2,3)
\pccurve[angleA=0,angleB=-135](2,3)(2.25,3.25)
\pccurve[angleA=45,angleB=180](2.25,3.25)(2.5,3.5)
\psline(2.5,3.5)(3.5,3.5)
\pccurve[angleA=0,angleB=90](3.5,3.5)(4,3)
\psline(4,3)(4,.5)
\pccurve[angleA=-90,angleB=0](4,.5)(3.75,0)
\pccurve[angleA=-90,angleB=180](3.5,.5)(3.75,0)
\pccurve[angleA=90,angleB=-45](3.5,.5)(3.28,.72)
\pccurve[angleA=135,angleB=-90](3.22,.78)(3,1)
\psline(3,1)(3,1.5)
\pccurve[angleA=90,angleB=-135](3,1.5)(3.25,1.75)
\pccurve[angleA=45,angleB=-90](3.25,1.75)(3.5,2)
\psline(3.5,2)(3.5,2.75)
\pccurve[angleA=90,angleB=0](3.5,2.75)(3.25,3)
\psline(3.25,3)(3,3)
\pccurve[angleA=180,angleB=45](3,3)(2.78,2.78)
\pccurve[angleA=-135,angleB=0](2.72,2.72)(2.5,2.5)
\pccurve[angleA=180,angleB=90](2.5,2.5)(2,2)
\pccurve[angleA=-90,angleB=45](2,2)(1.75,1.75)
\pccurve[angleA=-135,angleB=90](1.75,1.75)(1.5,1.5)
\psline(1.5,1.5)(1.5,1)
\pccurve[angleA=-90,angleB=135](1.5,1)(1.72,.78)
\pccurve[angleA=-45,angleB=90](1.78,.72)(2,.5)
\psline(2,.5)(2,.25)
\pccurve[angleA=-90,angleB=180](2,.25)(2.25,0)
\psline(2.25,0)(2.75,0)
\pccurve[angleA=0,angleB=-90](2.75,0)(3,.25)
\psline(3,.25)(3,.5)
\pccurve[angleA=90,angleB=-135](3,.5)(3.25,.75)
\pccurve[angleA=45,angleB=-90](3.25,.75)(3.5,1)
\psline(3.5,1)(3.5,1.5)
\pccurve[angleA=90,angleB=-45](3.5,1.5)(3.28,1.71)
\pccurve[angleA=135,angleB=-90](3.22,1.78)(3,2)
\psline(3,2)(3,2.5)
\pccurve[angleA=90,angleB=-45](3,2.5)(2.75,2.75)
\pccurve[angleA=135,angleB=-45](2.75,2.75)(2.5,3)
\pccurve[angleA=135,angleB=-45](2.5,3)(2.28,3.22)
\pccurve[angleA=135,angleB=0](2.22,3.28)(2,3.5)
\psline(2,3.5)(1.5,3.5)
\pccurve[angleA=180,angleB=45](1.5,3.5)(1.25,3.25)
\pccurve[angleA=-135,angleB=0](1.25,3.25)(1,3)
\pccurve[angleA=180,angleB=90](1,3)(.5,2.75)
\psline(.5,2.75)(.5,2.5)
\pccurve[angleA=-90,angleB=135](.5,2.5)(1.5,2)
\pccurve[angleA=-45,angleB=135](1.5,2)(1.72,1.78)
\pccurve[angleA=-45,angleB=90](1.78,1.72)(2,1.5)
\psline(2,1.5)(2,1)
\pccurve[angleA=-90,angleB=45](2,1)(1.75,.75)
\pccurve[angleA=-135,angleB=90](1.75,.75)(1.5,.5)
\psline(1.5,.5)(1.5,.25)
\pccurve[angleA=-90,angleB=0](1.5,.25)(1.25,0)
\psline(1.25,0)(.5,0)
\pccurve[angleA=180,angleB=-90](.5,0)(0,.5)
\rput[t](2,-.5){$(a)$} \rput[t](-.4,2){$7_5$}
\end{pspicture} \quad\quad
\begin{pspicture}[shift=-2](-.2,-.7)(4.2,3.7)
\psline(0,.5)(0,3) \pscircle[fillcolor=darkred, fillstyle=solid, linewidth=3pt](0,1.75){.1}
\pccurve[angleA=90,angleB=180](0,3)(.5,3.5)
\psline[arrowscale=1.5]{->}(0,2.4)(0,2.5)
\psline(.5,3.5)(1,3.5)
\pccurve[angleA=0,angleB=135](1,3.5)(1.22,3.28)
\pccurve[angleA=-45,angleB=180](1.28,3.22)(1.5,3)
\psline(1.5,3)(2,3) \psline[arrowscale=1.5]{->}(1.74,3)(1.76,3)
\pccurve[angleA=0,angleB=-135](2,3)(2.25,3.25)
\pccurve[angleA=45,angleB=180](2.25,3.25)(2.5,3.5)
\psline(2.5,3.5)(3.5,3.5) \psline[arrowscale=1.5]{->}(2.99,3.5)(3.01,3.5)
\pccurve[angleA=0,angleB=90](3.5,3.5)(4,3)
\psline(4,3)(4,.5) \psline[arrowscale=1.5]{->}(4,1.76)(4,1.74)
\pccurve[angleA=-90,angleB=0](4,.5)(3.75,0)
\pccurve[angleA=-90,angleB=180](3.5,.5)(3.75,0)
\pccurve[angleA=90,angleB=-45](3.5,.5)(3.28,.72)
\pccurve[angleA=135,angleB=-90](3.22,.78)(3,1)
\psline(3,1)(3,1.5) \psline[arrowscale=1.5]{->}(3,1.24)(3,1.26)
\pccurve[angleA=90,angleB=-135](3,1.5)(3.25,1.75)
\pccurve[angleA=45,angleB=-90](3.25,1.75)(3.5,2)
\psline(3.5,2)(3.5,2.75) \psline[arrowscale=1.5]{->}(3.5,2.37)(3.5,2.38)
\pccurve[angleA=90,angleB=0](3.5,2.75)(3.25,3)
\psline(3.25,3)(3,3)
\pccurve[angleA=180,angleB=45](3,3)(2.78,2.78)
\pccurve[angleA=-135,angleB=0](2.72,2.72)(2.5,2.5)
\pccurve[angleA=180,angleB=90](2.5,2.5)(2,2)
\pccurve[angleA=-90,angleB=45](2,2)(1.75,1.75)
\pccurve[angleA=-135,angleB=90](1.75,1.75)(1.5,1.5)
\psline(1.5,1.5)(1.5,1) \psline[arrowscale=1.5]{->}(1.5,1.26)(1.5,1.24)
\pccurve[angleA=-90,angleB=135](1.5,1)(1.72,.78)
\pccurve[angleA=-45,angleB=90](1.78,.72)(2,.5)
\psline(2,.5)(2,.25)
\pccurve[angleA=-90,angleB=180](2,.25)(2.25,0)
\psline(2.25,0)(2.75,0) \psline[arrowscale=1.5]{->}(2.49,0)(2.51,0)
\pccurve[angleA=0,angleB=-90](2.75,0)(3,.25)
\psline(3,.25)(3,.5) \psline[arrowscale=1.5]{->}(3,.37)(3,.38)
\pccurve[angleA=90,angleB=-135](3,.5)(3.25,.75)
\pccurve[angleA=45,angleB=-90](3.25,.75)(3.5,1)
\psline(3.5,1)(3.5,1.5) \psline[arrowscale=1.5]{->}(3.5,1.24)(3.5,1.26)
\pccurve[angleA=90,angleB=-45](3.5,1.5)(3.28,1.71)
\pccurve[angleA=135,angleB=-90](3.22,1.78)(3,2)
\psline(3,2)(3,2.5) \psline[arrowscale=1.5]{->}(3,2.24)(3,2.26)
\pccurve[angleA=90,angleB=-45](3,2.5)(2.75,2.75)
\pccurve[angleA=135,angleB=-45](2.75,2.75)(2.5,3)
\pccurve[angleA=135,angleB=-45](2.5,3)(2.28,3.22)
\pccurve[angleA=135,angleB=0](2.22,3.28)(2,3.5)
\psline(2,3.5)(1.5,3.5) \psline[arrowscale=1.5]{->}(1.76,3.5)(1.74,3.5)
\pccurve[angleA=180,angleB=45](1.5,3.5)(1.25,3.25)
\pccurve[angleA=-135,angleB=0](1.25,3.25)(1,3)
\pccurve[angleA=180,angleB=90](1,3)(.5,2.75)
\psline(.5,2.75)(.5,2.5)
\pccurve[angleA=-90,angleB=135](.5,2.5)(1.5,2)
\pccurve[angleA=-45,angleB=135](1.5,2)(1.72,1.78)
\pccurve[angleA=-45,angleB=90](1.78,1.72)(2,1.5)
\psline(2,1.5)(2,1) \psline[arrowscale=1.5]{->}(2,1.26)(2,1.24)
\pccurve[angleA=-90,angleB=45](2,1)(1.75,.75)
\pccurve[angleA=-135,angleB=90](1.75,.75)(1.5,.5)
\psline(1.5,.5)(1.5,.25) \psline[arrowscale=1.5]{->}(1.5,.38)(1.5,.37)
\pccurve[angleA=-90,angleB=0](1.5,.25)(1.25,0)
\psline(1.25,0)(.5,0) \psline[arrowscale=1.5]{->}(.88,0)(.87,0)
\pccurve[angleA=180,angleB=-90](.5,0)(0,.5)
\rput[t](-.4,2){$7_5$}
\rput(1.25,3.6){$1$} \rput(2.25,3.6){$2$} \rput(3.25,.45){$3$} \rput(3.25,1.45){$4$}
\rput(2.75,3.1){$5$} \rput(1.75,2){$6$} \rput(1.75,1){$7$}
 \rput(2.85,.75){$8$} \rput(3.65,1.75){$9$}
\rput(3.15,2.75){$10$}  \rput(2.25,2.8){$11$} \rput(1.25,2.8){$12$} \rput(2.2,1.75){$13$} \rput(2.2,.75){$14$} \rput[t](2,-.5){$(b)$}
\end{pspicture}
$$
\caption{$(a)$ The knot $7_5$ and $(b)$ the numbering of $7_5$ to find its Dowker-Thistlethwaite
(DT)-code.}
\label{DT75}
\end{figure}

Next, we explain the \emph{DT Code} (DT after C. Dowker and M. Thistlethwaite~\cite{DT})
of a knot $K$ because we will use it to identify prime knots whose flat plumbing basket number
is $6$. It can be obtained as follows:
\begin{enumerate}
\item[{\rm (1)}] Start walking along $K$ in a fixed direction
 as indicated in Fig.~\ref{DT75} $(a)$ and count every crossing you pass through.
 If $K$ has $n$ crossings and given that every crossing is visited twice, the count ends at $2n$. Label each crossing with the values of the counter when it is visited, though when labeling by an even number, take it with a minus sign if you are walking "over" the crossing.
\item[{\rm (2)}]
 Every crossing is now labeled with two integers whose absolute values run from $1$ to $2n$. It is easy to see that each crossing is labeled with one odd integer and one even integer. The DT code of $K$ is the list of even integers paired with the odd integers $1, 3, 5, \ldots$, taken in this order with the following modification: if the label is an even number and the strand followed crosses over at the crossing, then change the sign on the label to be a negative.
\item[{\rm (3)}]
 Thus for example the pairing for the knot in Fig.~\ref{DT75} $(b)$
 is given
 $$\left( \begin{matrix}
 1 & 3 & 5 & 7 & 9 & 11 & 13 \\
 12 & 8 & 10 & 14 & 4 & 2 & 6
 \end{matrix} \right)$$
 At last, the DT code for this labelling is the sequence $12~8~10~14~4~2~6$.
\end{enumerate}

From the given flat plumbing basket surface in Fig.~\ref{figure84band} $(a)$,
we obtain its DT code -$32$ -$14$ -$44$ -$22$ -$40$ $2$ $28$
$10$ -$48$ $18$ $36$ $6$ -$24$ -$42$ -$12$ $30$ -$16$ -$46$
-$20$ $34$ $4$ $26$ $8$ $38$ by choosing the starting point at the
upper-left corner of the rectangle and the clockwise orientation.

\section{The flat plumbing basket codes and results} \label{result}

For a given link $L$, Furihata \emph{et al.}~\cite{FHK:openbook} found an algorithm to find
a flat plumbing basket surface from a closed braid $\overline{\beta}=L$ as follows.

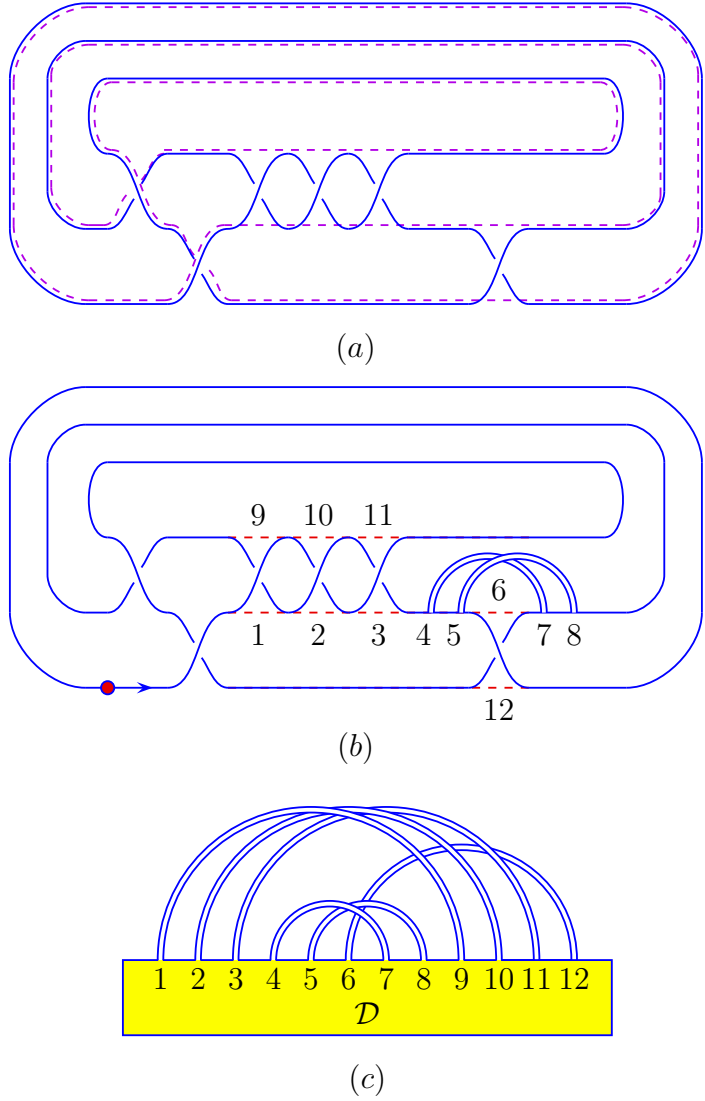
\begin{figure}
$$
\begin{pspicture}[shift=-1.2](0,-1.8)(9.5,3.1)
\psline[linecolor=pup, linestyle=dashed](.05,.05)(.05,1.95)
\pccurve[linecolor=pup, linestyle=dashed, angleA=90,angleB=180](.05,1.95)(1.05,2.95)
\psline[linecolor=pup, linestyle=dashed](1.05,2.95)(8.15,2.95)
\pccurve[linecolor=pup, linestyle=dashed, angleA=0,angleB=90](8.15,2.95)(9.15,2.05)
\psline[linecolor=pup, linestyle=dashed](9.15,2.05)(9.15,.05)
\pccurve[linecolor=pup, linestyle=dashed, angleA=-90,angleB=0](9.15,.05)(8.15,-.95)
\psline[linecolor=pup, linestyle=dashed](8.15,-.95)(2.95,-.95)
\pccurve[linecolor=pup, linestyle=dashed, angleA=180,angleB=-45](2.95,-.95)(2.65,-.55)
\pccurve[linecolor=pup, linestyle=dashed, angleA=135,angleB=0](2.45,-.35)(2.15,.05)
\pccurve[linecolor=pup, linestyle=dashed, angleA=180,angleB=0](2.15,.05)(1.35,1.05)
\pccurve[linecolor=pup, linestyle=dashed, angleA=180,angleB=180](1.35,1.05)(1.35,1.95)
\psline[linecolor=pup, linestyle=dashed](1.35,1.95)(7.85,1.95)
\pccurve[linecolor=pup, linestyle=dashed, angleA=0,angleB=0](7.85,1.95)(7.85,1.05)
\psline[linecolor=pup, linestyle=dashed](7.85,1.05)(2.05,1.05)
\pccurve[linecolor=pup, linestyle=dashed, angleA=180,angleB=45](2.05,1.000005)(1.75,.65)
\pccurve[linecolor=pup, linestyle=dashed, angleA=-135,angleB=0](1.55,.45)(1.25,.05)
\psline[linecolor=pup, linestyle=dashed](1.25,.05)(1.05,.05)
\pccurve[linecolor=pup, linestyle=dashed, angleA=180,angleB=-90](1.05,.05)(.55,.55)
\psline[linecolor=pup, linestyle=dashed](.55,.55)(.55,2.05)
\pccurve[linecolor=pup, linestyle=dashed, angleA=90,angleB=180](.55,2.05)(1.05,2.45)
\psline[linecolor=pup, linestyle=dashed](1.05,2.45)(8.15,2.45)
\pccurve[linecolor=pup, linestyle=dashed, angleA=0,angleB=90](8.15,2.45)(8.65,1.95)
\psline[linecolor=pup, linestyle=dashed](8.65,1.95)(8.65,.45)
\pccurve[linecolor=pup, linestyle=dashed, angleA=-90,angleB=0](8.65,.45)(8.15,.05)
\psline[linecolor=pup, linestyle=dashed](8.15,.05)(2.85,.05)
\pccurve[linecolor=pup, linestyle=dashed, angleA=180,angleB=0](2.85,.05)(2.05,-.95)
\psline[linecolor=pup, linestyle=dashed](2.05,-.95)(1.05,-.95)
\pccurve[linecolor=pup, linestyle=dashed, angleA=180,angleB=-90](1.05,-.95)(.05,.05)
\psline(0,0)(0,2)
\pccurve[angleA=90,angleB=180](0,2)(1,3)
\psline(1,3)(8.2,3)
\pccurve[angleA=0,angleB=90](8.2,3)(9.2,2)
\psline(9.2,2)(9.2,0)
\pccurve[angleA=-90,angleB=0](9.2,0)(8.2,-1)
\psline(8.2,-1)(6.9,-1)
\pccurve[angleA=180,angleB=-55](6.9,-1)(6.55,-.6)
\pccurve[angleA=125,angleB=0](6.45,-.4)(6.1,0)
\psline(6.1,0)(5.3,0)
\pccurve[angleA=180,angleB=-55](5.3,0)(4.95,.4)
\pccurve[angleA=125,angleB=0](4.85,.6)(4.5,1)
\pccurve[angleA=180,angleB=0](4.5,1)(3.7,0)
\pccurve[angleA=180,angleB=-55](3.7,0)(3.35,.4)
\pccurve[angleA=125,angleB=0](3.25,.6)(2.9,1)
\psline(2.9,1)(2.1,1)
\pccurve[angleA=180,angleB=55](2.1,1)(1.75,.6)
\pccurve[angleA=-125,angleB=0](1.65,.4)(1.3,0)
\psline(1.3,0)(1,0)
\pccurve[angleA=180,angleB=-90](1,0)(.5,.5)
\psline(.5,.5)(.5,2)
\pccurve[angleA=90,angleB=180](.5,2)(1,2.5)
\psline(1,2.5)(8.2,2.5)
\pccurve[angleA=0,angleB=90](8.2,2.5)(8.7,2)
\psline(8.7,2)(8.7,.5)
\pccurve[angleA=-90,angleB=0](8.7,.5)(8.2,0)
\psline(8.2,0)(6.9,0)
\pccurve[angleA=180,angleB=0](6.9,0)(6.1,-1)
\psline(6.1,-1)(2.9,-1)
\pccurve[angleA=180,angleB=-55](2.9,-1)(2.55,-.6)
\pccurve[angleA=125,angleB=0](2.45,-.4)(2.1,0)
\pccurve[angleA=180,angleB=0](2.1,0)(1.3,1)
\pccurve[angleA=180,angleB=180](1.3,1)(1.3,2)
\psline(1.3,2)(7.9,2)
\pccurve[angleA=0,angleB=0](7.9,2)(7.9,1)
\psline(7.9,1)(5.3,1)
\pccurve[angleA=180,angleB=0](5.3,1)(4.5,0)
\pccurve[angleA=180,angleB=-55](4.5,0)(4.15,.4)
\pccurve[angleA=125,angleB=0](4.05,.6)(3.7,1)
\pccurve[angleA=180,angleB=0](3.7,1)(2.9,0)
\pccurve[angleA=180,angleB=0](2.9,0)(2.1,-1)
\psline(2.1,-1)(1,-1)
\pccurve[angleA=180,angleB=-90](1,-1)(0,0)
\rput(4.6,-1.6){{$(a)$}}
\end{pspicture}$$
$$
\begin{pspicture}[shift=-1.2](0,-2)(9.5,3.1)
\psline[linecolor=darkred, linestyle=dashed](2.9,1)(6.9,1)
\psline[linecolor=darkred, linestyle=dashed](2.9,0)(6.9,0)
\psline[linecolor=darkred, linestyle=dashed](2.9,-1)(6.9,-1)
\psline(0,0)(0,2)
\pccurve[angleA=90,angleB=180](0,2)(1,3)
\psline(1,3)(8.2,3)
\pccurve[angleA=0,angleB=90](8.2,3)(9.2,2)
\psline(9.2,2)(9.2,0)
\pccurve[angleA=-90,angleB=0](9.2,0)(8.2,-1)
\psline(8.2,-1)(6.9,-1)
\pccurve[angleA=180,angleB=0](6.9,-1)(6.1,0)
\psline(6.1,0)(5.3,0)
\pccurve[angleA=180,angleB=-55](5.3,0)(4.95,.4)
\pccurve[angleA=125,angleB=0](4.85,.6)(4.5,1)
\pccurve[angleA=180,angleB=0](4.5,1)(3.7,0)
\pccurve[angleA=180,angleB=-55](3.7,0)(3.35,.4)
\pccurve[angleA=125,angleB=0](3.25,.6)(2.9,1)
\psline(2.9,1)(2.1,1)
\pccurve[angleA=180,angleB=55](2.1,1)(1.75,.6)
\pccurve[angleA=-125,angleB=0](1.65,.4)(1.3,0)
\psline(1.3,0)(1,0)
\pccurve[angleA=180,angleB=-90](1,0)(.5,.5)
\psline(.5,.5)(.5,2)
\pccurve[angleA=90,angleB=180](.5,2)(1,2.5)
\psline(1,2.5)(8.2,2.5)
\pccurve[angleA=0,angleB=90](8.2,2.5)(8.7,2)
\psline(8.7,2)(8.7,.5)
\pccurve[angleA=-90,angleB=0](8.7,.5)(8.2,0)
\psline(8.2,0)(6.9,0)
\pccurve[angleA=180,angleB=55](6.9,0)(6.55,-.4)
\pccurve[angleA=-125,angleB=0](6.45,-.6)(6.1,-1)
\psline(6.1,-1)(2.9,-1)
\pccurve[angleA=180,angleB=-55](2.9,-1)(2.55,-.6)
\pccurve[angleA=125,angleB=0](2.45,-.4)(2.1,0)
\pccurve[angleA=180,angleB=0](2.1,0)(1.3,1)
\pccurve[angleA=180,angleB=180](1.3,1)(1.3,2)
\psline(1.3,2)(7.9,2)
\pccurve[angleA=0,angleB=0](7.9,2)(7.9,1)
\psline(7.9,1)(5.3,1)
\pccurve[angleA=180,angleB=0](5.3,1)(4.5,0)
\pccurve[angleA=180,angleB=-55](4.5,0)(4.15,.4)
\pccurve[angleA=125,angleB=0](4.05,.6)(3.7,1)
\pccurve[angleA=180,angleB=0](3.7,1)(2.9,0)
\pccurve[angleA=180,angleB=0](2.9,0)(2.1,-1)
\psline(2.1,-1)(1,-1)
\pccurve[angleA=180,angleB=-90](1,-1)(0,0)
\pscircle[linecolor=blue, fillcolor=darkred, fillstyle=solid](1.3,-1){.1}
\psline[arrowscale=1.5]{->}(1.8,-1)(1.9,-1)
\psarc[doubleline=true](6.35,0){.75}{0}{180}
\psarc[doubleline=true](6.75,0){.75}{0}{180}
\rput(3.3,-.3){{$1$}}
\rput(4.1,-.3){{$2$}}
\rput(4.9,-.3){{$3$}}
\rput(5.5,-.3){{$4$}}
\rput(5.9,-.3){{$5$}}
\rput(6.5,.3){{$6$}}
\rput(7.1,-.3){{$7$}}
\rput(7.5,-.3){{$8$}}
\rput(3.3,1.3){{$9$}}
\rput(4.1,1.3){{$10$}}
\rput(4.9,1.3){{$11$}}
\rput(6.5,-1.3){{$12$}}
\rput(4.6,-1.8){{$(b)$}}
\end{pspicture}$$
$$
\begin{pspicture}[shift=-.8](-.7,-2.4)(6.2,2)
\psarc[doubleline=true](4,-.5){1.5}{-5}{185}
\psarc[doubleline=true](2.75,-.5){.75}{-5}{185}
\psarc[doubleline=true](2.25,-.5){.75}{-5}{185}
\psarc[doubleline=true](3,-.5){2}{-5}{185}
\psarc[doubleline=true](2.5,-.5){2}{-5}{185}
\psarc[doubleline=true](2,-.5){2}{-5}{185}
\psframe[linecolor=lightgray,fillstyle=solid,fillcolor=lightgray](-.5,-1.5)(6,-.5)
\psline(-.03,-.5)(-.5,-.5)(-.5,-1.5)(6,-1.5)(6,-.5)(4.53,-.5)
\psline(.03,-.5)(.47,-.5) \psline(.53,-.5)(.97,-.5) \psline(1.03,-.5)(1.47,-.5)
\psline(1.53,-.5)(1.97,-.5) \psline(2.03,-.5)(2.47,-.5) \psline(2.53,-.5)(2.97,-.5)
\psline(3.03,-.5)(3.47,-.5) \psline(3.53,-.5)(3.97,-.5) \psline(4.03,-.5)(4.47,-.5)
\psline(4.53,-.5)(4.97,-.5) \psline(5.03,-.5)(5.47,-.5)
\rput(0,-.75){{$1$}} \rput(.5,-.75){{$2$}}
\rput(1,-.75){{$3$}} \rput(1.5,-.75){{$4$}}
\rput(2,-.75){{$5$}} \rput(2.5,-.75){{$6$}}
\rput(3,-.75){{$7$}} \rput(3.5,-.75){{$8$}}
\rput(4,-.75){{$9$}} \rput(4.5,-.75){{$10$}}
\rput(5,-.75){{$11$}} \rput(5.5,-.75){{$12$}}
\rput(2.75,-1.2){{$\mathcal{D}$}}
\rput(2.75,-2.1){{$(c)$}}
\end{pspicture}
 $$
\caption{$(a)$ The knot $5_2$ as a closed braid, $(b)$ Seifert surface of $5_2$ to apply algorithm, $(c)$ a flat plumbing basket surface of $5_2$.} \label{52complete}
\end{figure}

\vskip .3cm
\noindent{\tt Algorithm}~\cite{FHK:openbook}
\begin{itemize}
\item $Step~1$. For a give link $L$, we find its braid representation $\beta$, the closed braid $\overline{\beta} =L$.
\item $Step~2$. Apply the method in~\cite{FHK:openbook} to obtain a flat plumbing basket surface $\F$
which is obtained from a disc by successively plumbing flat annuli.
\end{itemize}

This algorithm can be demonstrated in the following Example~\ref{exa1}.

\begin{example} \label{exa1}
A flat plumbing basket code of the knot $5_2$ is $(1$, $2$, $3$, $4$, $5$, $6$, $4$, $5$, $1$, $2$, $3$, $6)$.
\end{example}
\begin{proof}
For the knot $5_2$, we first present it as a closed braid $\overline{\sigma_2\sigma_1^{-1}(\sigma_2)^{-3}\sigma_1^{-1}}$
on three strings as illustrated in Fig.~\ref{52complete} $(a)$. Although theorem in~\cite{FHK:openbook} stated differently, one can choose
any two generators of the Artin's braid group $B_3$ as stated in Theorem~\ref{fpbktheoremext}.
We choose the first $\sigma_2\sigma_1^{-1}$ to have a disc $\mathcal{D}$
which is the union of three discs, bounded by three Seifert circles,
joined by two half twisted bands presented by $\sigma_2\sigma_1^{-1}$
as indicated by the dashed purple line in Fig.~\ref{52complete} $(a)$. Since the rest word
$(\sigma_2)^{-3}\sigma_1^{-1}$ has the length $4$ and $(\sigma_2)^{-3}$ has the different sign to $\sigma_2$. we need three flat plumbings.
However $\sigma_1^{-1}$ has the same sign to $\sigma_1^{-1}$, we first change the sign of half twisted band by adding two flat annuli as
shown in Fig.~\ref{52complete} (b). Now we pick as starting point as indicated as a red dot in Fig.~\ref{52complete} $(b)$. Then, we read
the flat bands along the disc $\mathcal{D}$ in the direction as given in Fig.~\ref{52complete} (b).
By isotoping the original disc $\mathcal{D}$ to a standard disc
as depicted as the rectangular gray region in Fig.~\ref{52complete} (c), we obtain a
flat plumbing basket surface $\F$ in Fig.~\ref{52complete} (c).
By rewriting labels in the set $\{7, 8, \ldots, 12\}$ by one in the set $\{1,2, \ldots, 6\}$
depend on how $1,2, \ldots, 6$ are connected to $7, 8, \ldots, 12$,
and by the rule that for annulus presented by $i$ is in front of the annulus presented by $j$ whenever $i>j$,
we obtain the first $12$-tuple $(1$, $2$, $3$, $4$, $5$, $6$, $4$, $5$, $1$, $2$, $3$, $6)$
which is the flat plumbing basket code of the flat plumbing basket surface $\F$.
\end{proof}

Let us remark that since two groups of annuli presented by $\{1,2,3\}$ and $\{4,5\}$ do not involve each other,
one may obtain a different flat plumbing basket code of the knot $5_2$, $(1$, $2$, $6$, $1$, $2$, $3$, $4$, $5$, $6$, $3$, $4$, $5)$
by changing (4; 5) into (1; 2) and (1; 2; 3) into (3; 4; 5).

One may see that the direction and the initial page of the open book decomposition was not fixed.
By moving the initial page from $1$ to $2$, the effect on the flat plumbing basket code is an action by
the permutation $\sigma=\left(\begin{matrix} 1 & 2 & \ldots & n-1 &n \\
2 & 3 & \ldots & n & 1 \end{matrix} \right)$.
By reversing the direction of pages in the open book decomposition,
the effect on the flat plumbing basket code is an action by
the permutation $\tau=\left(\begin{matrix} 1 & 2 & \ldots & n-1 &n \\
n & n-1 & \ldots & 2 & 1 \end{matrix} \right)$.
By summarizing these observations, we find the following theorem.

 \begin{theorem} \label{equivalent} For a positive integer $n$,
\begin{enumerate}
\item[{\rm (1)}]  The number of the $2n$-tuple flat plumbing basket codes presenting the same link are divisible by $n$.

\item[{\rm (2)}]  The number of the $2n$-tuple flat plumbing basket codes presenting the same link are divisible by $2$.
\end{enumerate}
\end{theorem}

Now we will provide some classification theorems of knots and links by
the flat plumbing basket numbers. Without using a computer program, the third author was able to prove
the following classification theorem of links of the flat plumbing basket number $0, 1, 2, 3$ and $4$ by using
slightly different presentations of the flat plumbing basket surfaces, known as \emph{permutation presentations}.

\begin{theorem} (\cite{Kim:link}) \label{linkclassupto4}
\begin{enumerate}
\item[{\rm (1)}]  A link $L$ has the flat plumbing basket number $0$ if and only if $L$ is the trivial knot.

\item[{\rm (2)}]  A link $L$ has the flat plumbing basket number $1$ if and only if $L$ is the trivial link of two components.

\item[{\rm (3)}]  A link $L$ has the flat plumbing basket number $2$ if and only if $L$ is the trivial link of three components.

\item[{\rm (4)}]  A link $L$ has the flat plumbing basket number $3$ if and only if $L$ is either the trivial link of four components or
the Hopf link which is denoted by $L2a1$.

\item[{\rm (5)}] A link $L$ has the flat plumbing basket number $4$ if and only if $K$ is
either the trefoil knot, the figure eight knot, $L2a1 \sqcup O$, $L2a1 \# L2a1$, $L6a5$, or the trivial link of five components.
\end{enumerate}
where $\sqcup$ presents the disjoint union.
\end{theorem}

It is fairly easy to see
that the number of components of the link whose
flat plumbing basket number $n$ is always congruent to $n+1$ modulo $2$.
Thus, the flat plumbing basket number of a prime knot has to be an even integer.
From Theorem~\ref{linkclassupto4}, one can easily obtain the following corollary
which addresses the classification of all prime knots of the flat plumbing basket number $0$,
$2$ and $4$.

\begin{corollary} \label{class4}
\begin{enumerate}
\item[{\rm (1)}]  A prime knot $K$ has the flat plumbing basket number $0$ if and only if $L$ is the trivial knot.

\item[{\rm (2)}]  There does not exist a prime knot $K$ whose flat plumbing basket number $2$.

\item[{\rm (3)}]  A knot $K$ has the flat plumbing basket number $4$ if and only if $K$ is either the trefoil knot
or the figure eight knot.
\end{enumerate}
\end{corollary}

\begin{example} \label{exa2}
\begin{enumerate}
\item[{\rm (1)}] The flat plumbing basket number of the link $4_1^2$ is $5$.
\item[{\rm (2)}] The flat plumbing basket number of the knots $5_2$ is $6$.
\end{enumerate}
\end{example}
\begin{proof}
Flat plumbing basket surfaces of the link $4_1^2$ with five annuli are depicted in Fig.~\ref{412complete}.
By Theorem~\ref{linkclassupto4}, the link $4_1^2$ can not have the flat plumbing basket number less than $5$.
Therefore, the flat plumbing basket number of the link $4_1^2$ must be $5$.

A flat plumbing basket surface of the link $5_2$ with six annuli are depicted in Fig.~\ref{52complete}
and its flat plumbing basket code is given $(1$, $2$, $3$, $4$, $5$, $6$, $4$, $5$, $1$, $2$, $3$, $6)$ in Example~\ref{exa1}.
By Corollary~\ref{class4}, the knot $5_2$ can not have the flat plumbing basket number less than $6$.
\end{proof}

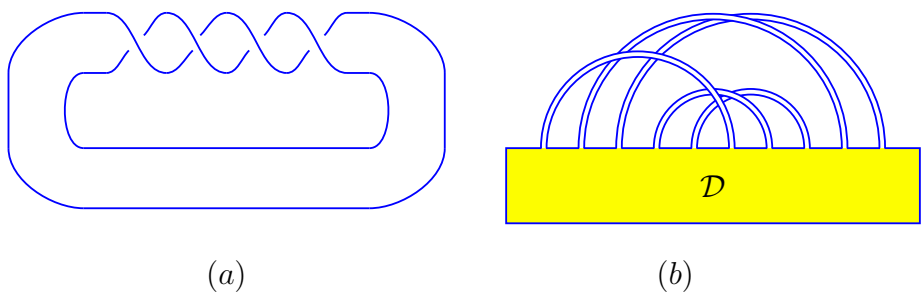
\begin{figure}
$$
\begin{pspicture}[shift=-1.2](-.2,-1.9)(6,2)
\psline(0,0)(0,1)
\pccurve[angleA=90,angleB=180](0,1)(1,1.8)
\psline(1,1.8)(1.3,1.8)
\pccurve[angleA=0,angleB=180](1.3,1.8)(2.1,1)
\pccurve[angleA=0,angleB=-135](2.1,1)(2.45,1.3)
\pccurve[angleA=45,angleB=180](2.55,1.5)(2.9,1.8)
\pccurve[angleA=0,angleB=180](2.9,1.8)(3.7,1)
\pccurve[angleA=0,angleB=-135](3.7,1)(4,1.3)
\pccurve[angleA=45,angleB=180](4.2,1.5)(4.5,1.8)
\psline(4.5,1.8)(4.8,1.8)
\pccurve[angleA=0,angleB=90](4.8,1.8)(5.8,1)
\psline(5.8,1)(5.8,0)
\pccurve[angleA=-90,angleB=0](5.8,0)(4.8,-.8)
\psline(4.8,-.8)(1,-.8)
\pccurve[angleA=180,angleB=-90](1,-.8)(0,0)
\pccurve[angleA=180,angleB=180](1,0)(1,1)
\psline(1,1)(1.3,1)
\pccurve[angleA=0,angleB=-135](1.3,1)(1.6,1.3)
\pccurve[angleA=45,angleB=180](1.8,1.5)(2.1,1.8)
\pccurve[angleA=0,angleB=180](2.1,1.8)(2.9,1)
\pccurve[angleA=0,angleB=-135](2.9,1)(3.2,1.3)
\pccurve[angleA=45,angleB=180](3.4,1.5)(3.7,1.8)
\pccurve[angleA=0,angleB=180](3.7,1.8)(4.5,1)
\psline(4.5,1)(4.8,1)
\pccurve[angleA=0,angleB=0](4.8,1)(4.8,0)
\psline(4.8,0)(1,0)
\rput(2.9,-1.7){{$(a)$}}
\end{pspicture} \quad
\begin{pspicture}[shift=-1.2](-.7,-1.9)(5.2,2)
\psarc[doubleline=true](2.75,0){.75}{-5}{185}
\psarc[doubleline=true](2.25,0){.75}{-5}{185}
\psarc[doubleline=true](2.75,0){1.75}{-5}{185}
\psarc[doubleline=true](2.25,0){1.75}{-5}{185}
\psarc[doubleline=true](1.25,0){1.25}{-5}{185}
\psframe[linecolor=lightgray,fillstyle=solid,fillcolor=lightgray](-.5,-1)(5,0)
\psline(-.03,0)(-.5,0)(-.5,-1)(5,-1)(5,0)(4.53,0)
\psline(.03,0)(.47,0) \psline(.53,0)(.97,0) \psline(1.03,0)(1.47,0)
\psline(1.53,0)(1.97,0) \psline(2.03,0)(2.47,0) \psline(2.53,0)(2.97,0)
\psline(3.03,0)(3.47,0) \psline(3.53,0)(3.97,0) \psline(4.03,0)(4.47,0)
\rput(2.25,-.5){{$\mathcal{D}$}}
\rput(1.75,-1.7){{$(b)$}}
\end{pspicture}
$$
\caption{$(a)$ The link $4_1^2$, $(b)$ a flat plumbing basket surface of the link $4_1^2$ with $5$ flat plumbings.} \label{412complete}
\end{figure}

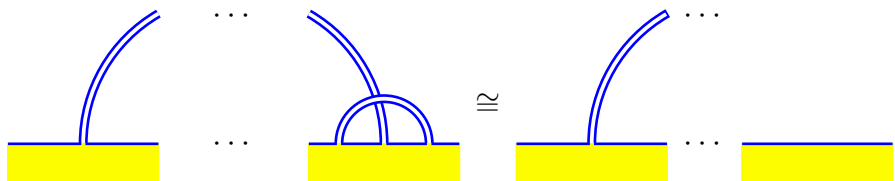
\begin{figure}
$$
\begin{pspicture}[shift=-1.2](-3.1,-.8)(3.1,2.2)
\psarc[doubleline=true, linewidth=1pt](0,0){2}{-10}{60}
\psarc[doubleline=true, linewidth=1pt](0,0){2}{120}{190}
\psarc[doubleline=true, linewidth=1pt](2,0){.6}{-10}{190}
\pspolygon[linecolor=lightgray, fillcolor=lightgray, fillstyle=solid](-1,0)(-3,0)(-3,-.5)(-1,-.5)(-1,0)
\pspolygon[linecolor=lightgray, fillcolor=lightgray, fillstyle=solid](1,0)(3,0)(3,-.5)(1,-.5)(1,0)
\psline[linewidth=1pt](-3,0)(-2.025,0)
\psline[linewidth=1pt](-1.97,0)(-1,0)
\psline[linewidth=1pt](1,0)(1.373,0)
\psline[linewidth=1pt](1.428,0)(1.972,0)
\psline[linewidth=1pt](2.03,0)(2.57,0)
\psline[linewidth=1pt](2.63,0)(3,0)
\rput(0,1.7){{$\cdots$}}
\rput(0,0){{$\cdots$}}
\end{pspicture}
 \cong
\begin{pspicture}[shift=-1.2](-3.1,-.8)(2.1,2.2)
\psarc[doubleline=true, linewidth=1pt](0,0){2}{120}{190}
\psline(1.94;120)(2.06;120)
\pspolygon[linecolor=lightgray, fillcolor=lightgray, fillstyle=solid](0,0)(2,0)(2,-.5)(0,-.5)(0,0)
\pspolygon[linecolor=lightgray, fillcolor=lightgray, fillstyle=solid](-1,0)(-3,0)(-3,-.5)(-1,-.5)(-1,0)
\psline[linewidth=1pt](-3,0)(-2.025,0)
\psline[linewidth=1pt](-1.97,0)(-1,0)
\psline[linewidth=1pt](0,0)(2,0)
\rput(-.5,1.7){{$\cdots$}}
\rput(-.5,0){{$\cdots$}}
\end{pspicture}
$$
\caption{A Type I move which decreases the flat plumbing basket number by $2$.} \label{twobridge}
\end{figure}

\begin{table}

\begin{tabular}{c|c|c} \hline
Name~of~knot & A~flat~plumbing~basket~codes & $\begin{matrix} \rm{Number~of} \\ \rm{f.p.b. codes} \end{matrix}$ \\ \hline
 $3_1$ & $\begin{matrix} \rm{12341234} \\ \rm{123124356456} \end{matrix}$  & $20,274$ \\
 $4_1$ & $\begin{matrix} \rm{12431243} \\ \rm{123124365465} \end{matrix}$ & $32,442$ \\
 $5_1$ & 123456123456 & $12$ \\
 $5_2$ & 123124563456 & $4,176$ \\
 $6_1$ & 123124653465 & $17,982$ \\
 $6_2$ & 123461253465 & $1,368$  \\
 $6_3$ & 123456123654 & $1,908$ \\
 $7_6$ & 123456123564 & $432$ \\
 $7_7$ & 123456124365 & $1,404$ \\
 $8_{1}$ & 124631254635 & $576$ \\
 $8_{3}$ & 125314625463 & $288$ \\
 $8_{12}$ & 123456124653 & $576$  \\
 $8_{20}$ & 123612546534 & $1,440$ \\
 $8_{21}$ & 123461256345 & $144$ \\
 $9_{42}$ & 123614235465 & $720$ \\
 $9_{44}$ & 123615246534 & $1,152$ \\
 $9_{46}$ & 123614534625 & $1,296$ \\
 $9_{48}$ & 123456125634 & $24$  \\
 $10_{132}$ & 123612564534 & $144$ \\
 $10_{136}$ & 125631243564 & $144$ \\
 $10_{137}$ & 124615346253 & $288$ \\
 $10_{140}$ & 123615624534 & $144$ \\
 $11n_{38}$ & 126154635423 & $144$ \\
 $12n_{462}$ & 124361546253 & $144$ \\
 $13n_{973}$ & 126415364253 & $144$ \\
 $14n_{17954}$ & 135264135264 & $36$  \\
 $15n_{45460}$ & 124635124635 & $216$ \\
 $16n_{246032}$ & 136254136254 & $72$ \\
 composite & & $2,268$ \\
 unknot & 123124563564 & $105,162$ \\ \hline
\end{tabular}
\vskip .2cm
\caption{Knots of the flat plumbing basket surfaces with $6$ bands
out of $195,120$ cases.}\label{t111}
\end{table}

\begin{theorem} \label{6classification}
The prime knot $K$ has the flat plumbing basket number $6$
if and only if it is either $5_1$, $5_2$, $6_1$, $6_2$,
$6_3$, $7_6$, $7_7$, $8_1$, $8_3$, $8_{12}$, $8_{20}$, $8_{21}$,
$9_{42}$, $9_{44}$, $9_{46}$, $9_{48}$, $10_{132}$,
$10_{136}$, $10_{137}$, $10_{140}$, $11n_{38}$,
$12n_{462}$, $13n_{973}$, $14n_{17954}$, $15n_{45460}$ or $16n_{246032}$.
\end{theorem}
\begin{proof}
To consider all flat plumbing basket surface with $6$ annuli,
we first count all possible flat plumbing basket codes of $\{ 1, 1, 2, 2, 3, 3, 4, 4, 5, 5, 6, 6\}$.
Since these are presented as circular shapes, we can fix the first elements to be $1$.
Thus, there are $\dfrac{11!}{2^6}=7,484,400$ many such flat plumbing basket codes.

The second author has written $C+$ program which determine whether a given flat plumbing basket code
produces a knot or a link with more than $1$ components, admits a Type I move as depicted in Fig.~\ref{twobridge}, finds
a DT-code of the given flat plumbing basket code and performs series of reductions which realize the Reidemeister
move type I and II. Using the program, we find that $6,415,200$ of them present links and
$1,069,200$ present knots. Among $1,069,200$ flat plumbing basket codes presenting knots, there are $874,080$ codes which
admit a Type I move, these knots must have the flat plumbing basket number less than $6$.
It leaves us $195,120$ flat plumbing basket codes. The computer
program ``knotfinder" of {\tt{Knotscape}} reduces these DT-codes
to the standard DT-codes and finds $105,162$ codes presents the unknot
and $2,268$ codes present composite knots.
All remaining $87,690$ reduced DT-codes of the flat plumbing basket codes are identified by the computer
program ``knotscape" of {\tt{Knotscape}}.
The results are listed as in Table~\ref{t111}.
\end{proof}

Let us remark that the $C+$ program is available at the first author's homepage :
{\tt{http://sec.pusan.ac.kr/?page$\b{\mbox{ }}$id=1303}}.
While using ``knotfinder" of {\tt{Knotscape}}, some of DT-codes of flat plumbing basket codes
were not identified since the program aborts by the time limit set up by the program
which may run infinitely because it repeatedly uses
Reidemeister moves on DT-code which may increase the length of DT-code.
However, we are lucky enough that at least one of
DT-codes of the flat plumbing basket codes which produce the same knot by
the action of $\sigma=\left(\begin{matrix} 1 & 2 & \ldots & n-1 &n \\
2 & 3 & \ldots & n & 1 \end{matrix} \right)$ as stated in Theorem~\ref{equivalent} $(1)$
makes an output DT-code. Further, because the length of these output DT-codes were less than
or equal to $16$ and all prime knots up to $16$ crossings are tabulated, we can find the exact knot names.
One can also observe that the number of the flat plumbing basket codes for a knot in Table~\ref{t111}
are all divisible by $12$ except $3_1$, $4_1$ and $6_1$ because some of flat plumbing basket codes
are invariant by the action of $\sigma$ or $\tau$.

Since the trefoil knot and the figure eight knot have the flat plumbing basket number $4$,
there are exactly $26$ prime knots whose flat plumbing basket numbers are exactly $6$.
The results in~\cite{HN} found the flat plumbing basket number of
prime knots up to $9$ crossings except $24$ knots.
Using Theorem~\ref{6classification}, we find the following corollary.

\begin{corollary} \label{cor9}
\begin{enumerate}
\item
The flat plumbing basket number of knots $7_2$, $7_4$ and $9_{45}$ is $8$.

\item
The flat plumbing basket number of knots $8_1$, $9_{44}$ is $6$.

\item
The flat plumbing basket number of knots $9_2$, $9_5$ and $9_{35}$ are either
$8$ or $10$.
\end{enumerate}
\end{corollary}

\section{Conclusion} \label{conclusion}

Authors already have all required DT-codes for the flat plumbing basket codes
of the flat plumbing basket number $8$ and $10$. But we are not able to find
a complete table like Table~\ref{t111}. There are two difficulties arise for
these cases 1) the computer program ``knotfinder" of {\tt{Knotscape}}
does not have a complete list of DT-codes for prime knots of more than $16$ crossings,
2) the number of DT-codes we have to deal with increases exponentially.
For example, a text.file contains all DT-codes for the flat plumbing basket codes
of the flat plumbing basket number $10$ is more than $2$-gigabytes already.

To solve these problems, authors are trying to write a new computer program
which identify whether two DT-notations are the same and revise ``knotfinder" of {\tt{Knotscape}}.

\section*{Acknowledgments}

The \TeX\, macro package PSTricks~\cite{PSTricks} was essential for
typesetting the equations and figures. The third author was supported by Basic Science Research Program through the
National Research Foundation of Korea(NRF) funded by the Ministry of Education, Science and Technology(2012R1A1A2006225).

\end{document}